\documentclass[12pt]{article}
\usepackage{pslatex}
\usepackage{amsmath,amsthm,amssymb,amsfonts, stmaryrd}
\usepackage{fancybox}
\usepackage{color, graphicx}
\usepackage[usenames,dvipsnames,svgnames,table]{xcolor}
 \usepackage[T1]{fontenc}
 \usepackage{hyperref}
\newtheorem{theorem}{Theorem}[section]

\newtheorem{conjecture}[theorem]{Conjecture}
\newtheorem{corollary}[theorem]{Corollary}

\newtheorem{definition}[theorem]{Definition}
\newtheorem{example}[theorem]{Example}

\newtheorem{lemma}[theorem]{Lemma}
\newtheorem{notation}[theorem]{Notation}

\newtheorem{proposition}[theorem]{Proposition}

\newtheorem{remark}[theorem]{Remark}

\newcommand{\Z}{\mathbb Z}

\newcommand{\N}{\mathbb N}

\DeclareMathOperator{\sd}{sdeg}

\begin{document}

\title{The Schur degree of additive sets}

\author{S. Eliahou and M.P. Revuelta}
\date{}

\maketitle

\begin{abstract} Let $(G,+)$ be an abelian group. A subset of $G$ is \emph{sumfree} if it contains no elements $x,y,z$ such that $x+y=z$. We extend this concept by introducing the \emph{Schur degree} of a subset of $G$, where Schur degree $1$ corresponds to sumfree. The classical inequality $S(n) \le R_n(3)-2$, between the Schur number $S(n)$ and the Ramsey number $R_n(3)=R(3,\dots,3)$, is shown to remain valid in a wider context, involving the Schur degree of certain subsets of $G$. Recursive upper bounds are known for $R_n(3)$ but not for $S(n)$ so far. We formulate a conjecture which, if true, would fill this gap. Indeed, our study of the Schur degree leads us to conjecture $S(n) \le n(S(n-1)+1)$ for all $n \ge 2$. If true, it would yield substantially better upper bounds on the Schur numbers, e.g. $S(6) \le 966$ conjecturally, whereas all is known so far is $536 \le S(6) \le 1836$.
\end{abstract}

\begin{quote}
Keywords: Sumfree; Schur numbers; Ramsey numbers; Discrete derivative; Minors. \vspace{0.15cm} \\ 
MSC Classification: 05D10, 11B75, 11P70
\end{quote}

\section{Introduction}

For $a,b \in \Z$, let $[a,b]=\{z \in \Z \mid a \le z \le b\}$ and $[a,\infty[=\{z \in \Z \mid a \le z\}$ denote the integer intervals they span. Denote $\N=\{0,1,2,\dots\}$ and $\N_+=\N\setminus \{0\}$. 

\smallskip
A subset of $\Z$ is \emph{sumfree} if it contains no elements $x,y,z$ such that $x+y=z$. The problem of partitioning $[1,N]$ into as few sumfree parts as possible was initiated by Schur~\cite{schur}. Given $n \in \N_+$, Schur established the existence of a number $S(n)$ such that $[1,N]$ can be partitioned into $n$ sumfree parts if and only if $N \le S(n)$. The $S(n)$ are called the Schur numbers and, despite more than a century in existence, remain poorly understood at the time of writing. Their only currently known values are
\begin{equation}\label{schur numbers}
\big(S(1),\,\, S(2),\,\, S(3),\,\, S(4),\,\, S(5)\big) = (1,\,\, 4,\,\, 13,\,\, 44,\,\, 160).
\end{equation}
See Section~\ref{sec comparisons} for more details. 
In his paper, Schur proved the following upper bound and recursive lower bound on the $S(n)$ for $n \ge 2$, namely
\begin{equation}\label{schur bounds}
3S(n-1)+1 \,\le\, S(n) \,\le\, n!e,
\end{equation}
leading in particular to $S(n) \ge (3^n-1)/2$ for all $n \ge 2$.

\smallskip
For $n \ge 1$, the $n$-color Ramsey number $R_n(3)=R(3,\dots,3)$ denotes the smallest $N$ such that, for any $n$-coloring of the edges of the complete graph $K_N$ on $N$ vertices, there is a monochromatic triangle. See~\cite{radzi} for an extensive dynamic survey on this topic. Only three of the numbers $R_n(3)$ are currently known, namely
\begin{equation}\label{ramsey numbers}
\big(R_1(3),\,\, R_2(3),\,\, R_3(3)\big) = (3,\,\, 6,\,\, 17).
\end{equation}
As for $n=4$, the presently known bounds are $51 \le R_4(3) \le 62$. It is conjectured in~\cite{XuR} that $R_4(3)$ equals $51$. Similarly to the upper bound in~\eqref{schur bounds}, it was shown in~\cite{GG} that $R_n(3) \le n!e+1$ for all $n \ge 1$. This bound has later been improved to $$R_n(3) \le n!(e-1/6)+1$$ for all $n \ge 4$ in~\cite{XuXC}. See also~\cite{E}, where the conjecture $R_4(3)=51$ is shown to imply $R_n(3) \le n!(e-5/8)+1$ for all $n \ge 4$.

\smallskip
In fact, there is a well known relationship between the Schur and the Ramsey numbers, namely
\begin{equation}\label{S and R v1}
S(n) \le R_n(3)-2.
\end{equation}
See e.g.~\cite{soifer}. That is, \emph{if the set $[1,N]$ admits a partition into $n$ sumfree parts, then $N \le R_{n}(3)-2$.} 

\smallskip

We shall show here that \eqref{S and R v1} holds in a more general context. Let $(G,+)$ be an abelian group. As in $\Z$, a subset of $G$ is \emph{sumfree} if it contains no elements $x,y,z$ such that $x+y=z$. Given a finite sequence $A=(a_1,\dots,a_N)$ in $G$, let us denote by $\hat{A}$ the set of all \emph{block sums} $a_{k}+\dots+a_\ell$ of $A$, where $1 \le k \le \ell \le N$. For instance, if $A=(1,\dots,1)$ of length $N$ in $G=\Z$, then $\hat{A}=[1,N]$. 

In this paper, we are concerned with partitioning subsets of $G$ of the form $\hat{A}$ into as few sumfree parts as possible. As just noted, this includes Schur's original problem for the integer intervals $[1,N]$.  Our extension of~\eqref{S and R v1} to this more general setting states that \emph{if $A$ is a sequence in $G$ of length $|A|=N$ and if $\hat{A}$ can be covered by $n$ sumfree parts, then $N \le R_{n}(3)-2$}.

Currently, the best available theoretical upper bound on $S(n)$ for $n \ge 4$ is the one provided by~\eqref{S and R v1}. While the Ramsey numbers $R_n(3)$ satisfy the well known recursive upper bound 
$$
R_n(3) \le n(R_{n-1}(3)-1)+2
$$
for all $n \ge 2$~\cite[Theorem 6, p. 6]{GG}, no similar statement is known yet for the $S(n)$. Here we fill this gap, at least conjecturally, as an outcome of our study of sumfree partitions of sets of the form $\hat{A}$. Indeed, as we shall see, that study leads us to conjecture the following recursive upper bound, for all $n \ge 2$:
\begin{equation}\label{conj S}
S(n) \le n(S(n-1)+1).
\end{equation}

The contents of this paper are as follows. In Section~\ref{sec basic}, we introduce the Schur degree and the basic notions and tools needed in the sequel. In Section~\ref{sec schur degree}, we prove initial properties of the Schur degree and illustrate them with selected examples in $\Z$. Our main result, an extension of \eqref{S and R v1} to sets $\hat{A}$ bounding their Schur degree with the Ramsey numbers $R_n(3)$, is proved in Section~\ref{sec comparison}. The material developed so far leads us in Section~\ref{sec conjectures} to the conjectural recursive upper bound~\eqref{conj S}, a substantial would-be improvement over~\eqref{S and R v1}.

\section{Basic notions and tools}\label{sec basic}

Here is the main notion introduced and studied in this paper.

\begin{definition} Let $(G,+)$ be an abelian group. Let $X \subseteq G$ be a subset. We define the \emph{Schur degree} of $X$, denoted $\sd(X)$, as the smallest $n \ge 1$ such that $X$ can be covered by $n$ sumfree subsets. If no such $n$ exists, we set $\sd(X)=\infty$.
\end{definition}
For instance, $\sd(X)=1$ if and only if $X$ is sumfree, whereas $\sd(X)=\infty$ whenever $0 \in X$, as $\{0\}$ is not sumfree. As another instance, in $\N$ we have
\begin{equation}\label{eq schur}
\sd([1,S(n)]) = n, \quad \sd([1,S(n)+1]) = n+1
\end{equation}
by definition of $S(n)$. Equivalently, $\sd([1,N]) \le n \iff N \le S(n)$.

\smallskip
Measuring the Schur degree of most subsets is likely to remain an extremely difficult task, even for the integer intervals $[1,N]$ as witnessed by the still highly mysterious Schur numbers $S(n)$. In this paper, we focus on subsets of a certain form $\hat{A}$, generalizing the intervals $[1,N]$ and introduced below.

\subsection{Block sums}

Let $(G,+)$ be an abelian group. Let $A=(a_1,\dots,a_N)$ be a finite sequence in $G$. We denote by $|A|=N$ its length and by $\sigma(A)=\sum_i a_i$ the sum of its elements. 

A \emph{block} in $A$ is any nonempty subsequence of consecutive elements of $A$. That is, any subsequence of the form 
\[
B=(a_{i},
\dots,a_j)
\] 
for some $1 \le i \le j \le N$. A \emph{block sum} in $A$ is a sum $\sigma(B)$ where $B$ is any block in $A$, i.e. any element in $G$ of the form $a_i+\dots+a_j$ for some $1 \le i \le j \le N$.

\begin{notation} Let $A=(a_1,\dots,a_N)$ be a sequence in $G$. We denote by 
$$\hat{A} = \{\sigma(B) \mid B \textrm{ is a block in }A \},$$ the set of block sums in $A$. 
\end{notation}
For instance, if $A=(1,\dots,1)$ of length $N$ in $\Z$, then $\hat{A}=[1,N]$ as noted above. In this paper, we initiate the study of the Schur degree of subsets of the form $\hat{A}$ for finite sequences $A$ in $G$, with the hope to shed some light on the basic case $[1,N]$ in $\Z$. Our main result is Theorem~\ref{main thm}, an extension of \eqref{S and R v1} to this context.

\subsection{Minors}

We show here that the association $A \mapsto \hat{A}$ is monotone with respect to taking minors, as defined below.

\begin{definition} Let $A=(a_1,\dots,a_N)$ be a sequence in the abelian group $G$.

\noindent
$\bullet$ An \emph{elementary contraction} of $A$ is any sequence $\overline{A}$ obtained by replacing a block $B$ in $A$ by its sum $\sigma(B)$. That is, if $B=(a_i,\dots,a_j)$ for some $1 \le i \le j \le N$, then
$$
\overline{A} \ = \ (a_1,\dots,a_{i-1},\sigma(B),a_{j+1},\dots,a_N).
$$
\noindent
$\bullet$ A \emph{contraction} of $A$ is any sequence obtained from $A$ by successive elementary contractions.
\end{definition}
For instance, let $A=(1,2,3,4)$. Then $(3,3,4)$, $(6,4)$ and $(3,7)$ are contractions of $A$, the first two ones being elementary. See also~\cite{ABERS}.

\begin{definition} Let $A=(a_1,\dots,a_N)$ be a sequence in $G$. A \emph{minor} of $A$ is either a block $B$ in $A$ or a contraction $\overline{A}$ of $A$.
\end{definition}

\begin{proposition}\label{prop minor} Let $G$ be an abelian group. Let $A$ be a finite sequence in $G$. If $B$ is a minor of $A$, then $\hat{B} \subseteq \hat{A}$.
\end{proposition}
\begin{proof}{}
The stated inclusion clearly holds if $B$ is a block in $A$, since any block sum of $B$ is a block sum of $A$. If $B$ is an elementary contraction of $A$ then again, any block sum of $B$ is a block sum of $A$. Therefore, the same holds if $B$ is obtained from $A$ by successive elementary contractions.
\end{proof}

\subsection{The discrete derivative}
For subsets $X,Y$ of a group $(G,+)$, their \emph{sumset} is $X+Y=\{x+y \mid x \in X, y \in Y\}$. Thus, $X$ is sumfree if and only if $(X+X) \cap X = \emptyset$; equivalently, if and only $(X-X) \cap X = \emptyset$, where $-X=\{-x \mid x \in X\}$.

In this section, for $X \subset \Z$ finite, we relate $X-X$ with a subset of the form $\hat{A}$ for a certain sequence $A$ closely linked to $X$. This is done with a variant of the discrete derivative, associating to a subset $X \subset \Z$ its sequence of successive jumps. See also~\cite{ABERS}.

\begin{definition} Let $X \subset \Z$ be a finite subset. Let the elements of $X$ be $x_0 < x_1 < \dots < x_r$. The \emph{discrete derivative} of $X$ is the sequence
\[
\Delta X \ = \ (x_1-x_0,\  x_2-x_1, \dots, \ x_r-x_{r-1})
\]
of successive jumps in $X$. 
\end{definition}

The interesting point for our purposes here is that $X-X$ can be read off from the block sums of $\Delta X$.

\begin{proposition}\label{X-X} Let  $X \subset \Z$ be a nonempty finite subset, and let $A=\Delta X$. Then
$$
\hat{A} = (X-X) \cap \N_+.
$$
\end{proposition}

\begin{proof} Denote by $x_0 < x_1 < \dots < x_r$ the elements of $X$. Then 
$$
(X-X) \cap \N_+ \ = \ \{x_t-x_s \mid 0 \le s < t \le r\}.
$$
Let $A = \Delta X = (a_1,\dots,a_r)$, where $a_i = x_i-x_{i-1}$ for $1 \le i \le r$. For any indices $0 \le s < t \le r$, let $B=(a_{s+1},\dots,a_t)$ be the corresponding block in $A$. Then
\begin{equation}\label{sigma B}
x_t-x_s=\sigma(B).
\end{equation}
Indeed, $\sigma(B)=\sum_{i=s+1}^t a_i= \sum_{i=s+1}^t (x_i-x_{i-1})=x_t-x_s$.
Hence $x_t-x_s \in \hat{A}$. This concludes the proof of the proposition.
\end{proof}

The next proposition bounds the Schur degree of certain subsets $\hat{A}$ in $\Z$. We start with a lemma.

\begin{lemma}\label{delta sumfree} Let $X$ be a sumfree subset of $[1, N]$ for some $N\in \N_+$. Let $A=\Delta(X)$. Then $\hat{A} \subseteq [1,N-1] \setminus X$.
\end{lemma}
\begin{proof} Denote $X=\{x_0,\dots,x_n\}$ with $1 \le x_0 < x_1 < \dots < x_n \le N$. Then $A=(a_1,\dots,a_n)$ where $a_i=x_i-x_{i-1}$ for all $1 \le i \le n$. Let $s \in \hat{A}$. Then 
$$s=a_i+\dots+a_j =x_j-x_{i-1}$$
for some $1 \le i \le j \le n$. Therefore $1 \le s \le N-1$, and $s \notin X$ since $s+x_{i-1}=x_j$ and $X$ is sumfree. That is, $s \in [1,N-1] \setminus X$, as desired.
\end{proof}

\begin{proposition}\label{prop sdeg} Let $N \ge 1$, and let $X_1 \sqcup \dots \sqcup X_n$ be a sumfree partition of $[1,N]$. Let $A_i=\Delta(X_i)$ for all $i$. Then $\sd(\widehat{A_i}) \le n-1$. 
\end{proposition}

\begin{proof} 
Let $i \in [1, n]$. It follows from Lemma~\ref{delta sumfree} that $\widehat{A_i}$ is contained in 
$$
X_1 \sqcup \dots \sqcup X_{i-1} \sqcup X_{i+1} \sqcup \dots \sqcup X_n.
$$
This induces a partition of $\widehat{A_i}$ into at most $n-1$ sumfree parts.
\end{proof}

\section{Basic properties of the Schur degree}\label{sec schur degree}

In this section, we compute the Schur degree in a few examples after giving its first basic properties. Let us start with the monotonicity of the Schur degree with respect to set inclusion.

\begin{lemma}\label{lem sdeg monotone} Let $G$ be an abelian group. If $X \subseteq Y\subseteq G$ then $\sd(X) \le \sd(Y)$.
\end{lemma}
\begin{proof} Let $n=\sd(Y)$. If $n= \infty$, we are done. Otherwise, $Y$ admits a partition into $n$ sumfree parts, inducing a partition of $X$ into at most $n$ sumfree parts.
\end{proof}
Here is a useful consequence.
\begin{proposition} Let $A$ be a finite sequence in the abelian group $G$. If $B$ is a minor of $A$, then $\sd(\hat{B}) \le \sd(\hat{A})$.
\end{proposition}
\begin{proof} 
We have $\hat{B} \subseteq \hat{A}$ by Proposition~\ref{prop minor}. Now apply Lemma~\ref{lem sdeg monotone}. 
\end{proof}
Note also that if $A'$ denotes the reverse sequence of $A$, then $\sd(\hat{A})=\sd(\widehat{A'})$. Indeed, $A$ and $A'$ have identical block sums, i.e. $\hat{A}=\widehat{A'}$. 

\smallskip
Our next proposition shows that the Schur degree is also monotone with respect to inverse images under group morphisms. We start with a lemma.

\begin{lemma}\label{lemma sumfree} Let $G_1,G_2$ be abelian groups and let $f \colon G_1 \to G_2$ be a morphism. Let $Y \subseteq G_2$. If $Y$ is sumfree then $f^{-1}(Y)$ also is.
\end{lemma}
\begin{proof} Assume that $f^{-1}(Y)$ is not sumfree. Then there exist $x_1,x_2,x_3 \in f^{-1}(Y)$ such that $x_1+x_2-x_3=0$. Hence $f(x_1)+f(x_2)-f(x_3)=0$, implying that $Y$ is not sumfree either.
\end{proof}

\begin{proposition}\label{map} Let $G_1,G_2$ be abelian groups and let $f \colon G_1 \to G_2$ be a morphism. Let $Y \subseteq G_2$. Then $\sd(f^{-1}(Y)) \le \sd(Y)$.
\end{proposition}
\begin{proof} 
Let $n=\sd(Y)$. Then there exist sumfree subsets $Y_1, \dots, Y_n \subseteq Y$ such that 
$$Y=Y_1 \sqcup\dots\sqcup Y_n.$$  
Therefore $f^{-1}(Y)=f^{-1}(Y_1)\sqcup\dots\sqcup f^{-1}(Y_n)$, and $f^{-1}(Y_i)$ is sumfree for all $i$ by Lemma \ref{lemma sumfree}. Hence $\sd(f^{-1}(Y))\le n$.
\end{proof}

\subsection{Examples}
As an illustration, we determine the Schur degree of a few selected subsets of $\Z$. In some cases, the results were obtained using specially written functions in \emph{Mathematica 10}~\cite{W}.

\begin{example} Let $B=[1,2] \cup [m,m+4]$. We claim that 
$$
\sd(B) = 3
$$
for all $m \ge 3$. Indeed, let $A=(1,1,m,1,1)$. Then $\hat{A}=B$, and $\sd(\hat{A}) \ge 3$ by Corollary~\ref{length 5} in the next section. Equality is obvious here.
\end{example}

\begin{example}\label{powers of 2}
Let $A = (2^i)_{0 \le i \le 13}$. 
Then here also, $\sd(\hat{A})=3$. But with one more term, i.e. for $B = (2^i)_{0 \le i \le 14}$, it is no longer the case as $\sd(\hat{B})=4$.
\end{example}

\begin{example}
This example is an application of Proposition~\ref{map}. Let $x,y$ be positive integers, and let $A=(x,y,\dots,x,y)$ be the $2$-periodic sequence of length $14$. Then $\sd(\hat{A} \setminus \{7x+7y\})=3$. Indeed, here are three sumfree classes covering that set:
\begin{align*}
C_1: &\,\,\, x,\, y,\, 2x+2y,\, 5x+5y,\, 7x+6y,\, 6x+7y. \\
C_2: &\,\,\, x+y,\, 2x+y,\, x+2y,\, 6x+5y,\, 5x+6y,\, 6x+6y.  \\
C_3: &\,\,\, 3x+2y,\, 2x+3y,\, 3x+3y,\, 4x+3y,\, 3x+4y,\, 4x+4y,\, 5x+4y,\, 4x+5y. 
\end{align*}

Mapping $x,y$ to $1$ yields a sumfree $3$-partition of $[1,13]$. In fact, the partition $C_1,C_2,C_3$ was constructed to do exactly that, using Proposition~\ref{map}.
\end{example}

\begin{example}\label{ex [1,6]} 
For each integer $x \ge 8$, one has
$$
\sd([1,6] \cup [x,x+13]) = 3.
$$
Indeed, this is shown by the following sumfree $3$-partition of this set:
\begin{align*}
C_1: &\,\,\,\, 1,6,x,x+3,x+7,x+10. \\
C_2: &\,\,\,\, 2,5,x+1,x+2,x+8,x+9.  \\
C_3: &\,\,\,\, 3,4,x+4,x+5,x+6,x+11,x+12,x+13. 
\end{align*}
However, adjoining $7$ to it, one has $\sd([1,7] \cup [x,x+13]) = 4$.
\end{example}

\begin{example}
Let $G$ be an abelian group containing $\Z$ and let $x \in G \setminus \Z$. Then
\begin{eqnarray*}
\sd(\{1,2\} \cup [x,x+3]) & = & 2, \\
\sd(\{1,2\} \cup [x,x+4]) & = & 3.
\end{eqnarray*}
Indeed, as easily seen, the \emph{only} sumfree $2$-coloring of $\{1,2\} \cup [x,x+3]$ is given by the two color classes $\{1,x,x+3\}$ and $\{2,x+1,x+2\}$. Hence, it is impossible to add $x+4$ to either class while maintaining the sumfree property.
\end{example}

\begin{example}
Let $G$ be an abelian group containing $\Z$. Let $x \in G \setminus \Z$ be such that $\{1,x\}$ is $\Z$-free, i.e. spans a free-abelian subgroup of rank $2$ of $G$. Then
$$
\sd([1,6] \cup (x+\N)) = 3.
$$
Indeed, consider the $3$-partition of Example~\ref{ex [1,6]} and extend it periodically as follows:
\begin{align*}
C_1: &\,\,\,\, 1,6,x,x+3,x+7,x+10, x+14,x+17,\dots \\
C_2: &\,\,\,\, 2,5,x+1,x+2,x+8,x+9,x+15,x+16,\dots  \\
C_3: &\,\,\,\, 3,4,x+4,x+5,x+6,x+11,x+12,x+13,x+18,x+19,x+20,\dots 
\end{align*}
One can also extend it towards the left. Thus in fact, $\sd([1,6] \cup (x+\Z)) = 3$. But here again, adjoining $7$ to it, one has $\sd([1,7] \cup (x+\Z))=4$.
\end{example}

\section{Comparison with $R_n(3)$}\label{sec comparison}

Recall that, for $n \ge 1$, the Ramsey number $R_n(3)$ denotes the smallest $N$ such that, for any $n$-coloring of the edges of the complete graph $K_N$, there is a monochromatic triangle. There is a well known relationship between the Schur and the  Ramsey numbers, namely
\begin{equation}\label{bound on S}
S(n) \le R_n(3)-2.
\end{equation}
Using the Schur degree of $[1,N]$, this may be expressed as follows:
$$
N \ge R_{n}(3)-1 \, \Longrightarrow \, \sd([1,N]) \ge n+1.
$$
Theorem~\ref{main thm} below extends this relationship to the Schur degree of $\hat{A}$ for any finite sequence $A$ in an abelian group. 

\begin{theorem}\label{main thm} Let $G$ be an abelian group. Let $A$ be a finite sequence in $G$. If $|A|\ge R_n(3)-1$ then $\sd(\hat{A}) \ge n+1$.
\end{theorem}

\begin{proof} Let $N=|A| \ge R_n(3)-1$. Denote $b(i,j)=x_i+ \dots +x_{j-1}$ for all $1 \le i < j \le N+1$. Then 
$$\hat{A}=\{b(i,j) \mid 1 \le i < j \le N+1\}.$$
Let $\chi \colon \hat{A} \to [1,n]$ be an arbitrary $n$-coloring of $\hat{A}$. Consider the complete graph $K_{N+1}=(V,E)$ on the vertex set $V=[1, N+1]$. Then $\chi$ induces an $n$-coloring $\chi' \colon E \to [1,n]$ on $E$ 
defined by 
$$\chi'(\{i,j\}) = \chi(b(i,j))$$ 
for all $1 \le i < j \le N+1$. Since $N+1 \ge R_n(3)$, there is a monochromatic triangle under $\chi'$ in $K_{N+1}$, say with vertices $i, j, h$ for some $1 \le i < j  <h \le N+1$. This yields, under $\chi$, the monochromatic subset 
$$\{b(i,j), b(j,h), b(i,h)\} \subset \hat{A}.$$
Since $b(i,j)+b(j,h)=b(i,h)$, the corresponding color class in $\hat{A}$ is not sumfree. Since $\chi$ was an arbitrary $n$-coloring of $\hat{A}$, we conclude that $\sd(\hat{A}) \ge n+1$.
\end{proof}

In particular, for $n=2$, $3$ and $4$, one has the following consequences.

\begin{corollary}\label{length 5} Let $A$ be a sequence in an abelian group $G$. If $|A|\ge 5$, then $\sd(\hat{A}) \ge 3$. If $|A|\ge 16$, then $\sd(\hat{A}) \ge 4$. If $|A|\ge 61$, then $\sd(\hat{A}) \ge 5$.
\end{corollary}

\begin{proof}
Follows from Theorem~\ref{main thm} and the well-known values $R_2(3) = 6$, $R_3(3) = 17$ and current upper bound $R_4(3) \le 62$.
\end{proof}

The converse of Theorem~\ref{main thm} does not hold in general. For instance, for $n=3$ and $A=(1,\dots,1)$ of length $14$ in $\Z$, by \eqref{eq schur} we have $\sd(\hat{A}) \ge 4$ since $\hat{A}=[1,14]$ and $S(3)=13$, yet $|A| \le R_3(3)-2=15$. However, here is a partial converse showing that Theorem~\ref{main thm} is best possible. First observe that if $|A|=N$, then 
$$|\hat{A}| \le 1+2+\cdots+N=\binom{N+1}2.$$
The case of equality, where all block sums in $A$ are pairwise distinct, is of interest. It occurs for instance if $A$ is \emph{$\Z$-free}, i.e. generates a subgroup isomorphic to $\Z^N$.

\begin{theorem}\label{thm converse} Let $A$ be a finite sequence in an abelian group $G$. If $|A| \le R_n(3)-2$ and $A$ is $\Z$-free, then $\sd(\hat{A}) \le n$.
\end{theorem}
\begin{proof} Denote $A = \{x_1,\dots,x_N\}$. Reusing the notation introduced in the proof of Theorem~\ref{main thm}, we have
$$\hat{A}=\{b(i,j) \mid 1 \le i < j \le N+1\}.$$
Again, let $K_{N+1}=(V,E)$ be the complete graph on the vertex set $V=[1, N+1]$. Consider the map $f \colon E \to \hat{A}$ 
defined by 
\begin{equation}\label{map f}
f(\{i,j\}) = b(i,j)
\end{equation}
for all $1 \le i < j \le N+1$. Since $|E|=|\hat{A}|$ and the $b(i,j)$ are pairwise distinct by assumption, the map $f$ is a \emph{bijection}. Since $N+1 \le R_n(3)-1$, there is an $n$-coloring $\chi \colon E \to [1,n]$ without any monochromatic triangle. Consider the composed map
$$
\chi \circ f^{-1} \colon \hat{A} \longrightarrow [1,n].
$$
We claim that  under this $n$-coloring of $\hat{A}$, every color class is sumfree. Indeed, let $u_1,u_2,u_3$ be any triple in $\hat{A}$ satisfying $u_1+u_2=u_3$. We claim that it cannot be monochromatic under $\chi \circ f^{-1}$. We have $u_1=b(i_1,j_1)$, $u_2=b(i_2,j_2)$, $u_3=b(i_3,j_3)$ for some indices $i_1 <j_1$, $i_2 <j_2$, $i_3 <j_3$ in $[1,N+1]$. The relation $u_1+u_2=u_3$ then becomes
$$
(x_{i_1}+\cdots+x_{j_1-1})+(x_{i_2}+\cdots+x_{j_2-1})=(x_{i_3}+\cdots+x_{j_3-1}).
$$
We may freely assume $i_1 \le i_2$. Since the sequence $x_1,\dots,x_N$ is $\Z$-free by hypothesis, the above equality is only possible if $i_1=i_3$, $j_1=i_2$ and $j_2=j_3$. That is, if the three edges $\{i_1,j_1\}$, $\{i_2,j_2\}$, $\{i_3,j_3\}$ form a triangle in $K_{N+1}$. Since that triangle is not monochromatic under $\chi$, the triple $u_1,u_2,u_3=u_1+u_2$ in $\hat{A}$ is not monochromatic under $\chi \circ f^{-1}$ either, since $f^{-1}(u_k)=\{i_k,j_k\}$ for $k=1, 2, 3$ by \eqref{map f}. Hence $\sd(\hat{A}) \le n$, as claimed.
\end{proof}

\begin{remark} The hypothesis that $A$ be $\Z$-free is not strictly needed in Theorem~\ref{thm converse}. For instance, let $A=(1,3, 3^2,\dots,3^{N-1})$. Even though $A$ is not $\Z$-free, it is still true that if $N \le R_n(3)-2$ then $\sd(\hat{A}) \le n$. This derives from the above proof and the fact that the only triples $u,v,u+v$ in $\hat{A}$ are those of the form $b(i,j),b(j,h),b(i,h)$. 
\end{remark}

\section{A recursive upper bound on $S(n)$?}\label{sec conjectures}

The Ramsey numbers admit well-known recursive upper bounds, including
\begin{equation}\label{ramsey bound}
R_n(3) \le n(R_{n-1}(3)-1)+2
\end{equation}
\cite[Theorem 6, p. 6]{GG}. To the best of our knowledge, no recursive upper bounds are known yet for the Schur numbers. We propose here a conjecture which, if true, would fill this gap. Let us start with an upper bound on $S(n)$ involving the number $L(n)$ defined below.

\begin{definition} Let $n \ge 2$. We define $L(n)$ to be the smallest positive integer with the following property: for every sequence $A$ in $\N_+$ of length $|A| = L(n)$ and average $\mu(A) \le n$, one has $\sd(\hat{A}) \ge n$. 
\end{definition}

\begin{example}\label{n=2} Let $n=2$. Then $L(2)=2$. Indeed, up to symmetry, the only sequences $A$ to consider are $(1,1)$, $(1,2)$, $(1,3)$, $(2,2)$. This yields $\hat{A}=\{1,2\}$, $\{1,2,3\}$, $\{1,3,4\}$, $\{2,4\}$, respectively. As none is sumfree, we have $\sd(\hat{A}) \ge 2$ in all cases, as required. 
\end{example}

Let us now establish the existence of $L(n)$ in full generality.
\begin{proposition}\label{prop L exists} For all $n\ge 2$, the number $L(n)$ exists and satisfies
\begin{equation}\label{L}
S(n-1)+1 \: \le\: L(n) \: \le\: R_{n-1}(3)-1.
\end{equation}
\end{proposition}

\begin{proof} If $A$ is any sequence in $\N_+$ of length $|A| = R_{n-1}(3)-1$, then irrespective of its average $\mu(A)$, we have $\sd(\hat{A}) \ge n$ by Theorem~\ref{main thm}, as desired. Thus $L(n)$ exists and is bounded above by $R_{n-1}(3)-1$. On the other hand, let $A=(1,\dots,1)$ of length $L(n)$ and average $\mu(A)=1$. Then $\hat{A}=[1,L(n)]$, whence $\sd([1,L(n)]) \ge n$ by hypothesis. Hence $L(n)  \ge S(n-1)+1$, by definition of $S(n-1)$.
\end{proof}

Here is our upper bound on $S(n)$ involving $L(n)$.
\begin{theorem}\label{conditional} We have $S(n) \le n L(n)$ for all $n \ge 2$.
\end{theorem}

\begin{proof} 
We claim that $[1,nL(n)+1]$ has Schur degree at least $n+1$. This will imply $nL(n)+1 \ge S(n)+1$, the desired conclusion. Assume for a contradiction that $nL(n)+1 \le S(n)$. Let then
\begin{equation}\label{partition}
[1,nL(n)+1] = X_1 \sqcup \dots \sqcup X_n
\end{equation}
be a sumfree partition. By the pigeonhole principle, one of the $X_i$'s has cardinality at least $L(n)+1$, say $|X_1| \ge L(n)+1$. Let $A=\Delta(X_1)$. Then $|A| \ge L(n)$, and $\sd(\hat{A}) \le n-1$ by Proposition~\ref{prop sdeg}. Let $B$ be a block of $A$ of length $|B|=L(n)$. Since $B$ is a minor of $A$, Proposition~\ref{prop minor} implies
\begin{equation}\label{sdeg B}
\sd(\hat{B}) \le \sd(\hat{A}) \le n-1.
\end{equation}
Let $s=\min(X_1)$, $t=\max(X_1)$. Then $\sigma(A) = t-s$ by~\eqref{sigma B}, and $t-s \le nL(n)$ since $X_1 \subseteq [1,nL(n)+1]$ by~\eqref{partition}. Hence $\sigma(B) \le nL(n)$ and so $\mu(B) = \sigma(B)/L(n) \le n.$ Since $|B|=L(n)$, the defining property of $L(n)$ implies $\sd(\hat{B}) \ge n$, contradicting~\eqref{sdeg B}. This concludes the proof of the theorem. 
\end{proof}

\begin{remark}
Proposition~\ref{prop L exists} and Theorem~\ref{conditional} imply the upper bound $$S(n) \le n(R_{n-1}(3)-1)$$
for all $n \ge 2$. However, this also follows by combining \eqref{bound on S} and \eqref{ramsey bound}, namely $S(n) \le R_n(3)-2$ and $R_n(3) \le n(R_{n-1}(3)-1)+2$.
\end{remark}

\subsection{Conjectures}

Given $n \ge 2$, what is the exact value of $L(n)$? It follows from Proposition~\ref{prop L exists} that 
\begin{equation}\label{if S+1= R-1}
\textit{ if }\ S(n-1)+1=R_{n-1}(3)-1 \ \textit{ then } \  L(n)=S(n-1)+1.
\end{equation}
This occurs for $n=2$ and $3$, since by~\eqref{schur numbers} and~\eqref{ramsey numbers}, we have $(S(1),R_1(3))=(1,3)$ and $(S(2),R_2(3))=(4,6)$. Thus $L(2)=2$ as already seen, and $L(3)=5$. As for $n=4$, we have
$$
(S(3),R_3(3))=(13,17).
$$
Proposition~\ref{prop L exists} then implies
$
14 \le L(4) \le 16.
$
We conjecture that $L(4)=14$ and, more generally, that the lower bound on $L(n)$ in~\eqref{L} is optimal. 

\begin{conjecture}\label{conj1} Let $n \ge 2$. Then $L(n)=S(n-1)+1$. That is, every sequence $A$ in $\N_+$ of length $|A|=S(n-1)+1$ and average $\mu(A) \le n$ satisfies \emph{$\sd(\hat{A}) \ge n$}.
\end{conjecture}

As shown below, this has very interesting consequences for the Schur numbers themselves.

\smallskip
We have seen above that Conjecture~\ref{conj1} holds for $n=2$ and $3$. Does it hold for $n=4$? That is, is it true that for any sequence $A$ in $\N_+$ of length $14$ and average $\mu(A) \le 4$, one has \emph{$\sd(\hat{A}) \ge 4$}? We do not know yet. In any case, some hypothesis bounding $\mu(A)$ from above cannot be completely dispensed of. For instance, consider the sequence
$$
A=(23,375,23,209,209,60,60,60,23,1,60,261,209,23)
$$
of length $14$. Then $|\hat{A}|=83$, and $\sd(\hat{A})=3$ as can be verified. 
But this does not contradict Conjecture~\ref{conj1} for $n=4$, since $\mu(A) =114$ here. Such exotic examples in length $14$ are hard to come by. This one was found with a semi-random search by computer. See also Example~\ref{powers of 2} with the powers of $2$, also of length $14$ but with a still higher average.

\medskip Here is a worthwhile consequence of Conjecture~\ref{conj1} for the Schur numbers, potentially the first known recursive upper bound for them.

\begin{conjecture}\label{conj2} $S(n) \le n(S(n-1)+1)$ for all $n \ge 2$.
\end{conjecture}

This directly follows from Theorem~\ref{conditional} and Conjecture~\ref{conj1}. Table~\ref{tab} shows that Conjecture~\ref{conj2} actually holds for $2 \le n \le 5$.

\medskip
\begin{table}[h!]
\centering
\begin{tabular}{|c||c|c|}
\hline
$n$ & $S(n)$ & $n(S(n-1)+1)$ \\
\hline \hline
1 & 1 &  \\
\hline \hline
2 & 4 & 4 \\ 
\hline \hline
3 & 13 & 15 \\ 
\hline \hline
4 & 44 & 56 \\ 
\hline \hline
5 & 160 & 225 \\
\hline
\end{tabular}
\caption{$S(n) \le n(S(n-1)+1)$ for $2 \le n \le 5$}
\label{tab}
\end{table}

\subsection{Comparisons}\label{sec comparisons}
Let us now compare this conjectural upper bound on $S(n)$ with the general currently known ones given by \eqref{bound on S} and \eqref{ramsey bound}, namely
\begin{equation}\label{S and R}
S(n) \le R_n(3)-2, \quad R_n(3) \le n(R_{n-1}(3)-1)+2.
\end{equation}

The currently known bounds on $R_4(3)$ are $51 \le R_4(3) \le 62$, established in \cite{Ch} and \cite{FKR}, respectively. Starting with $R_4(3) \le 62$, the bounds~\eqref{S and R} yield
$$
S(5) \le R_5(3)-2 \le 305, \quad S(6) \le R_6(3)-2 \le 1836.
$$

$\bullet$  For $n=4$, the equality $S(4)=44$ was established by computer~\cite{baumert}. But, as far as theory is concerned, nothing better than $S(4) \le R_4(3)-2 \le 60$ is currently known. A proof of Conjecture~\ref{conj1} for $n=4$ would yield $S(4) \le 56$, still far away from the true value $44$, yet a little closer to it.

\smallskip
$\bullet$ For $n=5$, the bound $S(5) \ge 160$ was first established in~\cite{exoo}, with  equality later conjectured to hold in \cite{FS}. Inded, the exact value $S(5)=160$ has recently been established by massive computer calculations with a certified SAT solver~\cite{heule}. A proof of Conjecture~\ref{conj1} for $n=5$, namely that every sequence $A$ in $\N_+$ such that $|A|=45$ and $\mu(A) \le 5$ satisfies $\sd(\hat{A}) \ge 5$, would imply $S(5) \le 225$. Here again, it would still be far away from the true value, yet it would provide a marked improvement over the currently best known theoretical upper bound $S(5) \le 305$.

\smallskip
$\bullet$ For $n=6$, on the one hand we have $S(6) \ge 536$ by~\cite{FS}, while at the time of writing, the best known upper bound is again the one given above, namely
$$
S(6) \le R_6(3)-2 \le 1836.
$$
By sharp contrast, using the true value $S(5)=160$, Conjecture~\ref{conj2} implies the following substantial improvement.

\begin{conjecture}\label{conj3} $S(6) \le 966$.
\end{conjecture}

$\bullet$  As for $n=7$, Conjectures~\ref{conj2} and~\ref{conj3} yield the conjectural upper bound $$S(7) \le 6769,$$ to be compared with the known ones given by~\eqref{S and R}, namely $S(7) \le R_7(3) \le 12861$. For a lower bound, the best we currently have is $S(7) \ge 1680$, by~\cite{FS} again.

\small

\bigskip

\textbf{Authors' addresses:} \vspace{-0.2cm}

\begin{itemize}
\item S.~Eliahou, Univ. Littoral C\^ote d'Opale, UR 2597 - LMPA - Laboratoire de Math\'ematiques Pures et Appliqu\'ees Joseph Liouville, F-62228
Calais, France and CNRS, FR2037, France. \vspace{-0.18cm}

e-mail: \url{eliahou@univ-littoral.fr} \vspace{-0.2cm}

\item M.P.~Revuelta, Departamento de Matem\'atica Aplicada I,
Universidad de Sevilla,
Avenida de la Reina Mercedes 4, 
C.P. 41012 Sevilla, Spain. \vspace{-0.18cm}

e-mail: \url{pastora@us.es}
\end{itemize}

\end{document}